\documentclass[oneside,10pt]{article}          
\usepackage{amsfonts,amsmath,latexsym,amssymb,amsthm} 
\usepackage{mathrsfs,upref}         
\usepackage{hyperref}
\usepackage{graphicx}
\usepackage [ all ]{xy}       	            
                	   
\newtheorem{theorem}{Theorem}           
\newtheorem{lemma}{Lemma}               
\newtheorem{corollary}{Corollary}

\newtheorem{prop}{Proposition}

\newtheorem{notation}{Notation}

\theoremstyle{definition}

\newtheorem{remark}{Remark}

\begin{document}

\title
{On a relationship between orthogonal projections and Toeplitz operators on poly-Bergman spaces of the upper half-plane: vertical symbols.}


\author{Maribel Loaiza, Miguel Antonio Morales-Ramos, \\
	 María del Rosario Ramírez-Mora and Josué Ramírez-Ortega}

%
%
%
%
%


\date{April 29, 2026}                               




\maketitle

\begin{abstract}
	In the context of studying $C^*$-algebras generated by Toeplitz operators acting on the poly-Bergman space $\mathcal{A}^2_{n}(\Pi)$  of the upper half-plane $\Pi$,
we introduce a system of all-but-one orthogonal projections in generic position.
We show that the $C^*$-algebra generated by these orthoprojections
is closely related to the $C^*$-algebra generated by all Toeplitz operators with vertical symbols satisfying 
boundary conditions.
This result suggests a new approach in the study of Toeplitz operators acting on other reproducing kernel Hilbert spaces.
Furthermore, the range of one of the orthoprojections herein has a reproducing kernel
expressed in terms of  the digamma and the Nielsen´s beta
functions. The harmonic function also emerges in  this development.
\end{abstract}




\section{Introduction}\label{sec1}

Although the study of algebras generated by orthogonal projections and those generated by Toeplitz operators have both received attention in the literature, they have largely been developed independently. Based on recent descriptions of $C^*$-algebras generated by Toeplitz operators, as well as on studies concerning orthogonal projections, we have identified a profound connection between these topics. The approach adopted in this work represents a fertile ground that deserves further exploration, given its potential to enhance the understanding of Toeplitz operators acting on various reproducing kernel Hilbert spaces with a wide range of symbols.

In this paper we focus on the poly-Bergman spaces $\mathcal{A}_{n}^2(\Pi)$ of the upper half-plane 
\[
\Pi = \{ z \in \mathbb{C} \ : \ \operatorname{Im} z > 0 \},
\]
and investigate a close connection between $C^*$-algebras generated by Toeplitz operators with vertical symbols and $C^*$-algebras generated by certain systems of all-but-one orthogonal projections in generic position. Let \(L_\infty^{\{0,\infty\}}(\mathbb{R}_+)\) denote the set of bounded measurable functions on \(\mathbb{R}_+\) that admit limits at 0 and at \(+\infty\), extended to \(\Pi\) as functions depending on the imaginary part of the complex variable. The \(C^*\)-algebra \(C^*(\mathcal{T}_n^{0,\infty})\), generated by Toeplitz operators with symbols in \(L_\infty^{\{0,\infty\}}(\mathbb{R}_+)\) acting on \(\mathcal{A}_{n}^2(\Pi)\), was characterized in \cite{J.R}. The subalgebra $\mathcal{C}^*(\mathcal{T}_n^{0,0})$ arises by restricting to symbols whose limits at \(0\) and \(+\infty\) coincide. It follows from \cite{J.R} that 
\[
\mathcal{C}^*(\mathcal{T}_n^{0,0}) \cong \mathcal{D}_n^{\mathbb{C}I},
\]
where $\mathcal{D}_n^{\mathbb{C}I}$ consists of all continuous \(M_n(\mathbb{C})\)-valued functions on \([0, \infty]\) whose values at \(0\) and \(+\infty\) reduce to scalar multiples of the identity matrix.

On the other hand, $C^*$-algebras generated by orthogonal projections in generic position have been described in different works (see for example \cite{Halmos}, \cite{Vasilevski} and \cite{projection}). The study of $C^*$-algebras generated by systems of all-but-one orthogonal projections in generic position has revealed structures that are remarkably close to $\mathcal{C}^*(\mathcal{T}_n^{0,0})$. Motivated by this observation, we consider the $C^*$-algebra 
\[
\mathcal{C}^*(P,\{B_{\Pi,(j)}\}_{j=1}^n),
\]
generated by the orthogonal projections $P$, $B_{\Pi,{(1)}}$,$\dots$, $B_{\Pi,{(n)}}$, where 
$B_{\Pi,{(j)}}$ denotes the orthogonal projection from $\mathcal{A}^2_{n}(\Pi)$ onto the 
true-$j$-poly-Bergman space $\mathcal{A}^2_{(j)}(\Pi)$, and $P$ is constructed  in terms of Toeplitz operators acting on the Bergman space of $\Pi$. Although there exist techniques to describe $C^*$-algebras generated by a system of all but one orthogonal projections in generic position (see, for example, \cite{projection}), in this work we use a noncommutative Stone-Weierstrass Theorem instead. As a result, we prove that the algebra \(\mathcal{C}^*(P,\{B_{\Pi,(j)}\}_{j=1}^n)\) is isometrically isomorphic to the algebra \(\mathcal{D}_n^{1,n}\), consisting of all \(M \in M_n(\mathbb{C}) \otimes C[0,\infty]\) such that \(M(0)\) and \(M(\infty)\) are diagonal matrices, and for all \(j = 2, \ldots, n-1\), the diagonal entries satisfy \(M_{jj}(0) = M_{jj}(\infty)\).

One of the main results of this work consists in the comparison between the algebra 
$\mathcal{C}^*(\mathcal{T}_n^{0,0})$ and the algebra 
$\mathcal{C}^*(P,\{B_{\Pi,(j)}\}_{j=1}^n)$. To better understand the behavior of $C^*$-algebras generated by Toeplitz operators, we study the  $C^*$-algebra generated by a single Toeplitz operator $T_{n,a_0}$ and the orthogonal projections $B_{\Pi,(1)},\dots,B_{\Pi,(n)}$. We show that $\mathcal{C}^*(T_{n,a_0}, B_{\Pi,(1)}, \dots, B_{\Pi,(n)})$
is isomorphic and isometric to the algebra of continuous functions on the positive real line with values in \(M_n(\mathbb{C})\), subject to the condition that they are diagonal at the points \(x=0\) and \(x=\infty\). All the algebras mentioned above are related as follows
$$\mathcal{C}^*(\mathcal{T}_n^{0,0}) \subset \mathcal{C}^*(P,\{B_{\Pi,(j)}\}_{j=1}^n) \subset \mathcal{C}^*(T_{n,a_0}, B_{\Pi,(1)}, \dots, B_{\Pi,(n)}).$$
We also have 
$$\mathcal{C}^*(\mathcal{T}_n^{0,\infty}) \subset \mathcal{C}^*(T_{n,a_0}, B_{\Pi,(1)}, \dots, B_{\Pi,(n)}).$$
It is worth mentioning that, for $n=2$, we have $\mathcal{C}^*(P,\{B_{\Pi,{(j)}}\}_{j=1}^2)= \mathcal{C}^*(T_{n,a_0}, \{B_{\Pi,{(j)}}\}_{j=1}^2)$, that is, the $C^*$-algebra $\mathcal{C}^*(\mathcal{T}_n^{0,\infty})$ is contained in the algebra generated by the orthogonal projections $P, B_{\Pi,(1)}$ and $B_{\Pi,(2)}$. Finally, a section of this work is devoted to the description of the image and an integral representations of the projection $P$.




\section{On the Stone-Weierstrass theorem for noncommutative $C^*$-algebras}

Our primary goal is to put forward an alternative in the study of $C^*$-algebras generated by 
Toeplitz operators acting on polyanalytic function spaces, which lead us to the task of describe
$C^*$-algebras generated by continuous matrix-valued functions on a compact Hausdorff space, 
where a Stone-Weierstrass theorem for noncommutative $C^*$-algebras becomes a fundamental tool.

Let us recall some facts about the general theory of $C^*$-algebras.
For the time being, let $\mathcal{A}$ denote a unital $C^*$-algebra.
For each  irreducible unitary representation $\pi : \mathcal{A} \rightarrow B(\mathcal{H})$, 
there is a collection $P_{\pi}(\mathcal{A})$ of pure states given by the general formula
$f_{\pi,v}(a)=\langle \pi(a)v,v \rangle_{\mathcal{H}}$ with $\|v\|_{\mathcal{H}}=1$;
and if $f_{\pi,v}=f_{\pi,w}$, then there exists a unimodular complex number $\lambda$
such that $w=\lambda v$. 
Now any irreducible representation of $\mathcal{A}$, unitarily equivalent to
$\pi$, generates the same collection of pure states associated to $\pi$.
Besides, each pure state  of $\mathcal{A}$ has the form
$f_{\pi,v}(a)=\langle \pi(a)v,v \rangle_{\mathcal{H}}$  for some irreducible representation $\pi$.
Furthermore, if $\rho$ is an irreducible representation of $\mathcal{A}$, which is not unitarily equivalent 
to $\pi$, then $P_{\rho}(\mathcal{A}) \cap P_{\pi}(\mathcal{A})=\emptyset$.
All these facts are consequences  of the Gelfand-Naimark-Segal (GNS) construction for $C^*$-algebras, 
for example see Section 5.1 in \cite{Murphy}, and Sections 2.4 and 2.5 in \cite{Dixmier}.

\vspace{.3cm}
Pure states play a fundamental role in Stone-Weierstrass's theorems for noncommutative $C^*$-algebras.
Suppose that $\mathcal{B}$ is a $C^*$-subalgebra of $\mathcal{A}$.
In the case of CCR (liminal) $C^*$algebras, I. Kaplansky proved that $\mathcal{B}=\mathcal{A}$ if the intersection $\mathcal{B}\cap (J\setminus K)$ is nonempty  for any two regular maximal right ideals $J$ and $K$ of $\mathcal{A}$ (see Theorem 7.2 in \cite{Kapla}). This statement holds for regular maximal left ideals.
Now the  map $\varphi \mapsto N_{\varphi}=\{ a\in \mathcal{A} \ : \  \varphi(a^*a)=0 \}$
is a one-to-one- correspondence from the set of pure states of $\mathcal{A}$
onto the set of regular maximal left ideals of $\mathcal{A}$ (Theorem 5.3.5 in \cite{Murphy}).
Then Kaplansky's theorem says:  $\mathcal{B}=\mathcal{A}$ if for any distinct pure states $\varphi$
and $\psi$ of $\mathcal{A}$, there exists a {\bf positive} element $b\in \mathcal{B}$
such that $\varphi(b)=0$ and $\psi(b)\neq 0$.
Certainly Glimm's contributions on Stone-Weierstrass's theorems deserve an outstanding atention \cite{Glimm};
the author remarks that the statement``$\mathcal{B}$ separates the set pure states of $\mathcal{A}$ is  a sufficient condition for $\mathcal{A}=\mathcal{B}$" follows from Kaplansky's paper.  
In this work the following theorem fits to our porpuse, and we refer to it as Kaplansky's theorem,
where positiveness on $b\in \mathcal{B}$ is an unnecessary condition.

\begin{theorem}[\cite{Dixmier}, pp. 252]
	Let $\mathcal{A}$ be a unital postliminal $C^*$-algebra, 
	and $\mathcal{B}$ a $C^*$-subalgebra of $\mathcal{A}$.
	Suppose that $\mathcal{B}$ separates the set of pure states of $\mathcal{A}$, that is, 
	for each distinct pure states $\varphi$ and $\psi$ of $\mathcal{A}$,
	there exists $b\in \mathcal{B}$ such that $\varphi(b) \neq \psi(b)$.
	Then $\mathcal{A}=\mathcal{B}$.
\end{theorem}

In particular, take $\mathcal{A}=M_n(\mathbb{C}) \otimes C(X)=M_n(C(X))$, where $X$ is a compact
Hausdorff space. In this case, $\mathcal{A}$ is the $C^*$-algebra of all continuous 
sections of the $C^*$-bundle $E=M_n(\mathbb{C}) \times X$, where
$p:E\rightarrow X$ is the canonical projection onto $X$.
Note that $M_n(\mathbb{C})$ and $C(X)$ are nuclear $C^*$-algebras,
thus there exists only one $C^*$-norm on $\mathcal{A}$.
For each $x_0 \in X$, the evaluation map $\pi_{x_0} : \mathcal{A} \rightarrow M_n(\mathbb{C})$
given by $\pi_{x_0}(M)=M(x_0)$ is an irreducible representation because it is surjective.
Then for each unimodular vector $v\in \mathbb{C}^n$ we have the pure state
given by
\begin{equation}\label{gen-pure-states}
	f_{x_0,v}(M)=\langle M(x_0)v,v \rangle.
\end{equation}
Let us see how this pure state can be obtained from another approach.
The $C^*$-algebra $M_n(\mathbb{C})$ has only one irreducible representation,
which is the natural identification with $B(\mathbb{C}^n)$, then all of its pure states have the
form $f_v(A)=\langle Av,v  \rangle$ with $\|v\|=1$, see also \cite{Lee}.
On the other hand, the space of pure states of $C(X)$ consists of all multiplicative functionals,
that is, it consists of all evaluation maps given by $F_{x_0}(f)=f(x_0)$  $\forall  f\in C(X)$.
It is easy to see that $f_v \otimes F_{x_0}=f_{x_0,v}$. Since $C(X)$ is abelian, every pure state
of $\mathcal{A}$ has the form $f_v \otimes F_{x_0}$, see Theorem 4.14 in \cite{Takesaki} pp. 211, and Theorem 6.4.13 in \cite{Murphy} pp. 204; according to the uniqueness up to unitary equivalence of the  GNS construction, the set of all $\pi_{x_0}$'s
exhausts the space of irreducible representations of $\mathcal{A}$.
Finally, by Urysohn's lemma, $\pi_{x_1}$ is not equivalent to $\pi_{x_0}$ whenever $x_1\neq x_0$.




\section{Bergman and poly-Bergman spaces}\label{sect-poly-bergman}

In this section we include some interesting results concerning poly-Bergman spaces of functions defined on the upper half-plane. To simplify notation, the identity operator is denoted by $I$  regardless of the space on which it acts. 
As usual, $\mathbb{C}$ denotes the complex plane with the area measure $dxdy$, and $\Pi $ stands for the upper half-plane. That is,
$$\Pi=\{z=x+iy\in \mathbb{C}:y>0\}.$$

A continuous function $f(z)=u(x,y)+iv(x,y):\Pi\to \mathbb{C}$ having partial derivatives up to order $n$
is called polyanalytic  of order $n$ if it belongs to  the kernel of the generalized Wirtinger  operator $(\partial / \partial \overline{z})^n $. Each polyanalytic function  $f$ of order $n$ can be written as
\begin{equation*}
	f(z)=\sum_{k=0}^{n-1}a_k(z)\overline{z}^k,
\end{equation*}
where $a_k$ is an analytic function for every $k=0,\dots, n-1$, see \cite{Balk}. 

The poly-Bergman space of order $n\neq 0$, denoted here by $\mathcal{A}^2_{n}(\Pi)$, is the closed subspace of $L_2(\Pi)$ consisting of all polyanalytic functions of order $n$. We denote by $\mathcal{A}^2_{0}(\Pi)$ the space consisting only of the zero function. It is clear that $\mathcal{A}^2_{0}(\Pi)\subset \mathcal{A}^2_{1}(\Pi)\subset \dots \subset \mathcal{A}^2_{n}(\Pi)$, then we consider the true-poly-Bergman spaces which are defined as follows:
$$
\mathcal{A}_{(n)}^2(\Pi)= \mathcal{A}_n^2(\Pi) \ominus \mathcal{A}_{n-1}^2(\Pi), \text{ for }n\neq 0.$$

Let $B_{\Pi,n}$ and $B_{\Pi,(n)}$ be the
orthogonal projections from $L_2(\Pi)$ 
onto $\mathcal{A}_n^2(\Pi)$ and
$\mathcal{A}_{(n)}^2(\Pi)$, 
respectively.

Now we include a unitary operator mapping poly-Bergman 
spaces into spaces of functions defined on the positive real line,  see \cite{PTB} for details.  There, the author considers the following unitary operators acting on 
$L_2(\Pi)=L_2(\mathbb{R}) \otimes L_2(\mathbb{R}_+)$:
$$U_1=F \otimes I,$$
$$(U_2\varphi)(x,y)= \frac{1}{\sqrt{2|x|}}\varphi\left(x,\frac{y}{2|x|}\right),$$
where $F$ is the Fourier transform on $L_2(\mathbb{R})$ given by 
$$F (f)(t)=\frac{1}{\sqrt{2\pi}}\int_{-\infty}^{\infty} f(x) e^{-itx} dx.$$
The unitary operator $U=U_2U_1$ transforms the poly-Bergman space of order $n$ onto the set of  all functions of the form 
\begin{equation*}
	\label{decomposition2}
	\sum_{k=0}^{n-1}\chi_+(x)f_k(x)\ell_k(y),\end{equation*}
where $\chi_+$ is the indicator function of $\mathbb{R}_+$,  $f_k$ belongs to $L_2(\mathbb{R})$, and  $\ell_k$ is the normalized Laguerre function
\begin{equation}\label{elek}
	\ell_{k}(y)= (-1)^k \chi_+(y) e^{-y/2}\frac{e^y}{k!}\frac{d^k}{dy^k}(e^{-y}y^{k}), \, k=0,1,2,\dots .
\end{equation}

For $k=0,1, \dots ,$ let $L_k$ denote the one-dimensional space generated by $\ell_k$, and consider the direct sum
$$L_n^\oplus= \bigoplus_{k=0}^n L_k.$$

\begin{theorem}
	\label{Relation}
	The unitary operator $U=U_2U_1$ gives an isometric isomorphism
	of the space $L_2(\Pi)$, under which
	\begin{enumerate}
		\item[1)]  the true-poly-Bergman space $\mathcal{A}_{(n)}^2(\Pi)$  
		is mapped onto the space
		$L_2(\mathbb{R}_+) \otimes L_{n-1}$,
		
		\item[2)] the poly-Bergman space $\mathcal{A}_{n}^2(\Pi)$ 
		is mapped 
		onto the space
		$L_2(\mathbb{R}_+) \otimes L_{n-1}^\oplus$.
	\end{enumerate}
\end{theorem}

By the last theorem, we can identify each true-poly-Bergman space $\mathcal{A}^2_{(n)}(\Pi)$ with $L_2(\mathbb{R}_+)$. To be precise, let $R_{0,(n)}: L_2(\mathbb{R}_+) \rightarrow L_2(\Pi)$ be the operator given by the rule
$$(R_{0,(n)} f)(x,y)= \chi_+(x)f(x)\ell_{n-1}(y).$$
According to Theorem \ref{Relation} $1)$, the operator
$$R_{(n)}:=R_{0,(n)}^{*}U$$
isometrically maps  $\mathcal{A}_{(n)}^2(\Pi)$ 
onto $L_2(\mathbb{R}_+)$. Further
\begin{align*}
	R_{(n)}^{ *}R_{(n)} &= B_{\Pi,(n)} :L_2(\Pi) \rightarrow \mathcal{A}_{(n)}^2(\Pi),\\
	R_{(n)} R_{(n)}^{*} &= I: L_2(\mathbb{R}_+)\rightarrow L_2(\mathbb{R}_+). 
\end{align*}

From now on, for a Hilbert space ${\cal H}$, let
$${\cal H}^n=\{h= (h_1,....,h_n)^t \ : \ h_j\in {\cal H} \}$$
with the usual inner product, and the superscript denotes the transpose matrix.

Introduce the  embedding $R_{0,n}:(L_2(\mathbb{R}_+))^n \rightarrow L_2(\Pi)$ given by the formula
\begin{align*}\label{def:Rn} 
	(R_{0,n} f)(x,y)
	&= \chi_+(x)[N_n(y)]^t f(x),
\end{align*}
where $f=(f_1,\dots,f_n)^t$ and $N_n(y)=(\ell_0(y),\dots,\ell_{n-1}(y))^t.$
 It is easy to prove (see \cite[pp.~1807--1808]{J.R}) that  $R_{0,n}^*:L_2(\Pi) \rightarrow (L_2(\mathbb{R}_+))^n$ 
is given by
\begin{equation}\label{R*n}
	(R_{0,n}^* \varphi)(x)= \chi_+(x) \int_{\mathbb{R}_+} \varphi(x,y)N_n(y)dy,   
\end{equation}
and that the operator 
\begin{equation}\label{mapA_Ln}
	R_n:= R_{0,n}^{*}U
\end{equation} 
isometrically maps $\mathcal{A}_n^2(\Pi)$ onto $(L_2(\mathbb{R}_+))^n$.  
Thus,
\begin{align}
	R_n^{*}R_n &= B_{\Pi,n} :L_2(\Pi) \rightarrow \mathcal{A}_n^2(\Pi),   \label{Rn*Rn}\\
	R_n R_n^{*}&=I:(L_2(\mathbb{R}_+))^n\rightarrow (L_2(\mathbb{R}_+))^n.   \label{RnRn*}
\end{align} 

Combining the results in {\cite[Corollary~2.5, Theorem~2.4]{PTB}} with 
the construction given in \cite[pp.~1807--1808]{J.R} we obtain
the explicit formulae for $R_n$ and $R_n^*$, which are used in the upcoming sections. They are given by 
\begin{equation}\label{Rn} 
	(R_{n} \varphi)(x)= \frac{\chi_+(x)}{\sqrt{2\pi}} \int_{\mathbb{R}} e^{-xti}
	\int_{\mathbb{R}_+} \varphi(t,r) \sqrt{2|x|} N_n(2|x|r) dr dt, 
	\quad \quad \varphi \in L_2(\Pi) \end{equation}
and
\begin{equation}
	\label{Restrella}(R_{n}^* f)(x,y)= \frac{1}{\sqrt{2\pi}} \int_{\mathbb{R}_+} e^{xti}
	\sqrt{2t} N_n(2ty)^t f(t) dt, 
	\quad  \quad  f \in (L_2(\mathbb{R}_+))^n.\end{equation}

Note that the integral in \eqref{R*n} as well as the integral in  \eqref{Rn} refer to the vector integral, which means that we are integrating each component individually.




\section{Toeplitz operators with vertical symbols}\label{sect-Toeplitz}

In \cite{SQV}, the authors characterized all commutative $C^*$-algebras generated by Toeplitz operators acting on the Bergman space of the unit disk $\mathbb{D}$, or alternatively, on the upper half-plane $\Pi$. Each of these algebras is generated by Toeplitz operators whose symbols are invariant under the action of a maximal abelian group of M\"obius transformations acting on $\Pi$. One of these groups is the group of real numbers $\mathbb{R}$ acting on the upper half-plane by horizontal translations. That is, each $h\in \mathbb{R}$ induces the M\"obius transformation
$$z\mapsto z+h.$$
Those symbols invariant under the above action of $\mathbb{R}$ are precisely all functions that  depend only on the imaginary component of the complex variable $z$. They are called vertical symbols and have been used in different settings to describe $C^*$-algebras generated by the corresponding Toeplitz operators (see \cite{J.R},\cite{J.R2} and \cite{vasi} for example).

In this work we use Toeplitz operators acting on different spaces. In order to prevent any confusion, below we explain the notation that we use throughout this work. 
For $n\in \mathbb{N}$ and a bounded measurable symbol $a$, we denote by $T_{(n),a}$ the Toeplitz operator with symbol $a$ acting on $ \mathcal{A}_{(n)}^2(\Pi)$ by the rule $T_{(n),a}(f)=B_{\Pi,(n)}(af)$.
Finally, $T_{n,a}$ stands for the Toeplitz operator acting on $\mathcal{A}_{n}^2(\Pi)$. 
It is well known that,  if $a$ is a bounded vertical symbol, the Toeplitz operator $T_{(n),a}$ is unitarily equivalent 
to the multiplication operator $\gamma^{(n),a}I=R_{(n)}T_{(n),a}R_{(n)}^{*}$
acting on $L_2(\mathbb{R}_+)$ (\cite[Theorem 3.2]{J.R}), where
\begin{equation}\label{gamma-(0)-a}
	\gamma^{(n),a}(x)
	=\int_{\mathbb{R}_+}a \left(\frac{y}{2|x|}\right) \left(\ell_{n-1}(y)\right)^2 dy.
\end{equation}
Analogously, the Toeplitz operator $T_{n,a}$ is unitarily equivalent to the matrix-multiplication operator $\gamma^{n,a}I=R_nT_{n,a}R_n^{*}:(L_2(\mathbb{R}_+))^n\to (L_2(\mathbb{R}_+))^n$ (\cite[Theorem 3.3]{J.R}), where
\begin{equation}\label{gamma-n-a}
	\gamma^{n,a}(x)
	=(\gamma^{n,a}_{jk}(x))_{j,k=1,\dots,n}=\int_{\mathbb{R}_+}a \left(\frac{y}{2|x|}\right) N_n(y)[N_n(y)]^t dy,
\end{equation} 
and $N_n(y)=(\ell_0(y),\dots, \ell_{n-1}(y))^t$.

In case $n=1$, we use the standard notation for a Toeplitz operator,  namely, $T_{(1),a}$ is denoted by $T_{a}$, and $\gamma^{(1),a}$  by $\gamma^{a}$ (see \cite[Theorem 5.2.1]{vasi}).

In order to get boundary conditions for the matrix-valued function given in \eqref{gamma-n-a}, in \cite{J.R} the authors considered a subclass of vertical symbols, denoted by $L_\infty^{\{0,\infty\}}(\mathbb{R}_+)$, and consisting of all functions in $L_\infty(\mathbb{R}_+)$ having limit values 
at the points $0$  and $+\infty$. That is,
for each $a \in  L_\infty^{\{0,\infty\}}(\mathbb{R}_+)$ the following 
limits exist
$$a^0:=\lim_{y\rightarrow 0^+}a(y), \quad  \quad 
a^{+\infty}:=\lim_{y \rightarrow +\infty}a(y).$$
Here every function $a \in  L_\infty^{\{0,\infty\}}(\mathbb{R}_+)$ is identified 
with its extension $a(z):=a(y)$ to the upper half-plane $\Pi$, where
$y= \mathrm{Im}\, z$. Elementary calculations show that
\begin{equation*}\label{Prop:Gamma}
	a^{+\infty} I= \lim_{x\rightarrow 0^+}\gamma^{n,a}(x), \quad
	a^{0} I=  \lim_{x\rightarrow +\infty}\gamma^{n,a}(x).
\end{equation*}

For a family $\mathcal{F}$ of bounded linear operators acting on a Hilbert space, 
the $C^*$-algebra generated by $\mathcal{F}$ is denoted by $\mathcal{C}^*(\mathcal{F})$.

\begin{theorem}
	\label{CAToeplitz1}
	Let $\mathcal{T}_n^{0,\infty}=\{T_{n,a}: a\in  L_\infty^{\{0,\infty\}}(\mathbb{R}_+)\}$. 
	The $C^*$-algebra $\mathcal{C}^*(\mathcal{T}_n^{0,\infty})$ is isomorphic and isometric to the 
	$C^*$-subalgebra $\mathcal{D}_n^{0,\infty}$ of 
	$M_n(\mathbb{C}) \otimes C[0,\infty]$ consisting of all matrix-valued functions $M$
	such that $M(0)$ and $M(\infty)$ belong to $\mathbb{C} I$. 
	The isomorphism from $\mathcal{C}^*(\mathcal{T}_n^{0,\infty})$ onto $\mathcal{D}_n^{0,\infty}$
	is given on the generators by the rule
	$$T_{n,a} \longmapsto \gamma^{n,a}.$$
\end{theorem}

Denote by  $\mathcal{T}_n^{0,0}$ the set of all Toeplitz operators
$T_{n,b}$ with symbols $b \in L_\infty^{\{0,\infty\}}(\mathbb{R}_+)$ satisfying the condition $b^0=b^{+\infty}$.
The proof of the following corollary is similar to that of Theorem \ref{CAToeplitz1}.

\begin{corollary}\label{CAToeplitz} 
	The $C^*$-algebra $\mathcal{C}^*(\mathcal{T}_n^{0,0})$ is isomorphic and isometric to 
	the $C^*$-algebra $\mathcal{D}_n^{\mathbb{C}I}$, which consists of all
	matrix-valued functions $M\in M_n(\mathbb{C}) \otimes C[0,\infty]$ such that  $M(0)=M( \infty)$. 
\end{corollary}




\section{Toeplitz operators and  orthogonal projections}\label{sect-PGeneric}

In \cite{Douglas, Halmos, Pedersen, Vasilevski} and \cite{projection}, the authors delve into the description of $C^*$-algebras generated by orthogonal projections.
Building upon these works, we observed that certain $C^*$-algebras generated 
by Toeplitz operators acting on $\mathcal{A}_{n}^2(\Pi)$ can alternatively be studied through a system of all-but-one 
orthogonal projections in generic position. Following \cite{projection}, the orthogonal projections from 
$\mathcal{A}_{n}^2(\Pi)$ onto the true-polyBergman subspaces are natural candidates for our purpose; 
of course, one more orthogonal projection must be constructed.
We overcome this difficulty by means of an isomorphism from $\mathcal{A}_{n}^2(\Pi)$ onto $(\mathcal{A}^2(\Pi))^n$.
Thus, a system of all-but-one orthogonal projections in generic position acting on $(\mathcal{A}^2(\Pi))^n$
is constructed below, one of orthoprojections is given  in terms of Toeplitz operators acting on the Bergman space $\mathcal{A}^2(\Pi)$.
As a consequence,  the $C^*$-algebra generated by such orthoprojections  is related to a 
$C^*$-algebra generated by Toeplitz operators acting on $\mathcal{A}_{n}^2(\Pi)$, see Corollary \ref{CAToeplitz}.

According to  \cite{projection},  we say that
all-but-one orthogonal projections $P, P_1, \dots, P_n$ are in generic position if
$$(\text{Im}\, P)^{\perp} \cap \text{Im}\, P_j=\{0\}=\text{Im}\, P \cap (\text{Im}\, P_j)^{\perp},    \quad j=1,\dots,n.$$

The following theorem provides the general method (up to isomorphism) for constructing a system of all-but-one orthogonal projections in generic position.

\begin{theorem}
	\label{Vas-generic} 
	Let $L$ be a Hilbert space. Suppose that $C_k: L \longrightarrow L$ are injective positive operators for 
	$k=1,...,n$ which satisfy
	$C_jC_k=C_kC_j$ and
	$$\sum_{k=1}^nC_k=I.$$
	Then the set of operators
	\begin{itemize}
		\item[i)] $P:=(C_j^{\frac{1}{2}}C_k^{\frac{1}{2}})_{j,k=1,\dots,n}: L^n \longrightarrow L^n$,
		\item[ii)] $P_j:={\rm diag }(0,\dots,I,\dots,0): L^n \longrightarrow L^n, \quad  j=1, \dots, n,$
	\end{itemize}
	is a system all-but-one orthogonal projections in generic position.
\end{theorem}

We take $L=\mathcal{A}^2(\Pi)$.
For $j=1,\dots,n$, let 
$$Q_{\Pi,(j)}: (\mathcal{A}^2(\Pi))^n \rightarrow (\mathcal{A}^2(\Pi))^n$$
be  the orthogonal projection given by
\begin{equation}\label{QPij}
	Q_{\Pi,(j)}f=(0,\dots,f_j,\dots,0)^t,
\end{equation}
where
$f=(f_1,\dots,f_n)^t \in (\mathcal{A}^2(\Pi))^n$. The orthogonal projections given above are quite natural and hold the following properties
\begin{align*}
	Q_{\Pi,(j)}Q_{\Pi,(k)}&=0 \ \text{ if } j\neq k,\\
	\sum_{j=1}^{n}Q_{\Pi,(j)}&=I.
\end{align*}

We  proceed to the construction of an orthoprojection
$P_T : (\mathcal{A}^2(\Pi))^n \rightarrow (\mathcal{A}^2(\Pi))^n$ according to  Theorem \ref{Vas-generic}. To do so,
introduce the Toeplitz operators $T_{a_k}: \mathcal{A}^2(\Pi) \rightarrow \mathcal{A}^2(\Pi)$ 
with vertical symbols 
\begin{equation}\label{vertical-aj}
	a_n(z)=\chi_{[n-1,\infty)}(y) \quad  \text{and} \quad a_k(z)=\chi_{[k-1,k)}(y) \ \text{for} \ k=1,\dots,n-1.
\end{equation}

Consider the operator  $P_T: \, (\mathcal{A}^2(\Pi))^n \rightarrow (\mathcal{A}^2(\Pi))^n$ defined by the formula
\begin{equation}\label{PT}
	P_T=\left( \sqrt{T_{a_j}} \sqrt{T_{a_k}}\right)_{j,k=1,\dots,n}.
\end{equation}

\begin{lemma}\label{cor:Vas-generic} 
	The set of operators $P_T, Q_{\Pi,(1)},\dots, Q_{\Pi,(n)}$ is a system of all-but-one orthogonal projections in generic position. 
\end{lemma}

\begin{proof}
	Equation \ref{gamma-(0)-a} implies that the Toeplitz operator 
	$T_{a_k}: \mathcal{A}^2(\Pi) \rightarrow \mathcal{A}^2(\Pi)$ is 
	unitarily equivalent to $\gamma^{a_k}I: L_2(\mathbb{R}_+) \rightarrow L_2(\mathbb{R}_+)$, where
	\begin{equation}\label{eq.gk}
		\gamma^{a_k}(x)=\int_{2(k-1)x}^{2kx}(\ell_0(y))^2 dy =\int_{2(k-1)x}^{2kx}e^{-y} dy=e^{-2kx}\left[-1+e^{2x}\right] \quad    
	\end{equation}
	for $k=1,\dots,n-1$, and
	\begin{equation}\label{eq:gn}
		\gamma^{a_n}(x)= \int_{2(n-1)x}^{\infty} e^{-y}dy=e^{-2(n-1)x}.
	\end{equation}
	Each Toeplitz operators $T_{a_k}$ is injective and positive since  the 
	multiplication operator $\gamma^{a_k}I$ holds these properties. 
	From the equality $a_1(z)+a_2(z)+\dots+a_n(z)=1$ it follows that
	$$T_{a_1}+T_{a_2}+\dots+T_{a_n}=I.$$
	Theorem \ref{Vas-generic} completes the proof.
\end{proof}

\vspace{0.3cm}
We now move on to describing the $C^*$-algebra $\mathcal{C}^*(P_T;\{Q_{\Pi,(j)}\}_{j=1}^n)$, we start by introducing a Bargmann-type transform that establishes an identification between the spaces $(\mathcal{A}^2(\Pi))^n$  and $(L_2(\mathbb{R}_+))^n$. 
We introduce the operator
\begin{equation*}\label{def:R}
	\mathbf{R}=\bigoplus_{k=1}^{n}R_1: (L_2(\Pi))^n \rightarrow (L_2(\mathbb{R}_+))^n,
\end{equation*}
where $R_1$ is given in Equation \ref{mapA_Ln}.  Recall that $R_1$ isometrically maps 
$\mathcal{A}^2(\Pi)$  onto $L_2(\mathbb{R}_+)$, 
this is the meaning of equations \eqref{Rn*Rn} and \eqref{RnRn*} for $n=1$.
Then the operator $\mathbf{R}$ holds the following properties
\begin{align*}
	\mathbf{R}^{ *}\mathbf{R}&= \bigoplus_{k=1}^{n}B_{\Pi,1} :(L_2(\Pi))^n \rightarrow (\mathcal{A}^2(\Pi))^n,   \\
	\mathbf{R} \mathbf{R}^{*}&= I: (L_2(\mathbb{R}_+))^n\rightarrow (L_2(\mathbb{R}_+))^n.   
\end{align*}
Then the restriction of $\mathbf{R}$ to $(\mathcal{A}^2(\Pi))^n$  is a unitary operator. 
Thus, we can identify each bounded operator acting on $(\mathcal{A}^2(\Pi))^n$ with a bounded operator acting on $(L_2(\mathbb{R}_+))^n$ through the map
\begin{equation}\label{iso}
	T\mapsto \mathbf{R} T\mathbf{R}^*.
\end{equation}

In particular, the orthogonal projection given in \eqref{QPij} is unitarily equivalent to
$$ \mathbf{R} Q_{\Pi,(j)} \mathbf{R}^*=Q_j,$$
where $Q_j$ is the orthogonal projection, acting on $(L_2(\mathbb{R}_+))^n$, and given by  
\begin{equation}\label{Qj}
	Q_j(f_1,\dots,f_n)^t=(0,\dots,f_j,\dots,0)^t.
\end{equation}
In what follows the operator $Q_j$ will be identified with the $n\times n$ matrix whose entries are zero except
the $(j,j)$-entry which is $1$.
Of course $Q_j Q_k =\delta_{jk} Q_j$ and $Q_1+\cdots+Q_{n}=I$.

Let $P_\gamma I$ be the orthogonal projection defined by $P_\gamma I:=\mathbf{R} P_T \mathbf{R}^{*}$. Note that $\gamma^{a_j}I=R_{1}T_{a_j}R_{1}^{*}$, then 
$R_1 \sqrt{T_{a_j}}\sqrt{T_{a_k}} R_1^*=\sqrt{\gamma^{a_j}}\sqrt{\gamma^{a_k}} I$.
Thus
\begin{equation}\label{P-gamma}
	P_\gamma  I=\mathbf{R} P_T \mathbf{R}^{*}=\left( \sqrt{\gamma^{a_j}}\sqrt{\gamma^{a_k}} \right)_{j,k=1,\dots,n}I
	\in B((L_2(\mathbb{R}_+))^n).
\end{equation}

\begin{corollary}\label{cor} 
	The set of operators $P_\gamma I, Q_1,\ldots, Q_n$ is a system all-but-one orthogonal projections in generic position. 
\end{corollary}



\subsection{The $C^*$-algebra generated by $P_T$ and $Q_{\Pi,(1)},\dots, Q_{\Pi,(n)}$ }\label{subsect-alg-PT}

In this section we have a two-fold purpose.  
First, the operators $P_{\gamma}I,Q_1,...,Q_n$ are identified with continuous matrix-valued functions on 
the two-point compactification of the positive real line.
Instead of using the techniques given in \cite{Douglas}, \cite{Halmos}, \cite{Pedersen}, \cite{Vasilevski} and \cite{projection},
we describe the $C^*$-algebra $\mathcal{C}^*(P_{\gamma};\{Q_j\}_{j=1}^n)$ using the noncommutative Stone-Weierstrass theorem. 
As a consequence, we get a description up to isomorphism of the $C^*$-algebra  $\mathcal{C}^*(P_T;\{Q_{\Pi,(j)}\}_{j=1}^n)$,
which is a subalgebra of $B((\mathcal{A}^2(\Pi))^n)$.
Secondly, an isomorphism from $(\mathcal{A}^2(\Pi))^n$ onto $\mathcal{A}_{n}^2(\Pi)$ allow us to
construct a system of all-but-one orthogonal projections $P,B_{\Pi,(1)},...,B_{\Pi,(n)}$ in generic position acting on the poly-Bergman space  $\mathcal{A}_{n}^2(\Pi)$.
At the end of the section  we make a comparison between $\mathcal{C}^*(P;\{B_{\Pi,(j) }\}_{j=1}^n)$ and the $C^*$-algebra generated by Toeplitz operators with vertical symbols in  $L_\infty^{0,\infty}(\mathbb{R}_+)$.

According to the isomorphism (\ref{iso}),
the $C^*$-algebra $\mathcal{C}^*(P_T;\{Q_{\Pi,(j)}\}_{j=1}^n)$ is  isomorphic to 
the $C^*$-algebra $\mathcal{C}^*(P_\gamma I;\{Q_j\}_{j=1}^n)$. 
In turn, this $C^*$-algebra is identified with the $C^*$-algebra $\mathcal{B}^{1,n}$ generated by the matrix-valued functions
$Q_1,...,Q_n$ and  $P_{\gamma}=\left( \sqrt{\gamma^{a_j}}\sqrt{\gamma^{a_k}} \right)_{j,k=1,\dots,n}$ , where
$Q_j$ also denotes the constant matrix-valued function whose components are zero except the $(j,j)$-entry which equals $1$.
Of course these matrices are continuous on the interval $[0,\infty]$.
Equations \ref{eq.gk} and \ref{eq:gn} imply that $\gamma^{a_1}(0)=\gamma^{a_n}(\infty)=0$, $\gamma^{a_n}(0)=\gamma^{a_1}(\infty)=1$ and
$\gamma^{a_k}(0)=\gamma^{a_k}(\infty)=0$ for $k=2,\dots,n-1$. That is, 
$$P_{\gamma} (0)=E_{nn} 
\qquad \text{and}  \qquad
P_{\gamma}(+\infty)=E_{11},$$
where $E_{jk}$ stands for the  $n \times n$-matrix whose components are zero, except for the $(j,k)$-entry which equals 1. 

\begin{notation}
	Let $\mathcal{D}_n^{1,n}$ be the $C^*$-algebra consisting of all matrices $M\in M_n(\mathbb{C})\otimes C[0,\infty]$ such that $M(0),M(\infty)$ are diagonal, and $M_{jj}(0)=M_{jj}(\infty)$ for all $j=2,\dots,n-1$.
\end{notation}

Note that the generators of $\mathcal{B}^{1,n}:={\mathcal{C}}^*(P_{\gamma};\{Q_j\}_{j=1}^n)$ belong to the algebra $\mathcal{D}_n^{1,n}$, thus $\mathcal{B}^{1,n}\subset \mathcal{D}_n^{1,n}$. 
To prove that these algebras coincide we use the noncommutative Stone-Weierstrass  theorem for postliminal $C^*$-algebras.
Now every liminal $C^*$-algebra is postliminal, and every $C^*$-subalgebra of a postliminal $C^*$-algebra is postliminal (\cite{Dixmier}).
Since $\mathcal{A}:=M_n(\mathbb{C})\otimes C[0,\infty]$ is liminal, the $C^*$-algebra $\mathcal{D}_n^{1,n}$ is postliminal.
Actually $\mathcal{D}_n^{1,n}$ is liminal. Of course, we have to identify all pure states of $\mathcal{D}_n^{1,n}$ in order to 
prove that $\mathcal{B}^{1,n}=\mathcal{D}_n^{1,n}$. 
For every pure state $\varphi$ of $\mathcal{D}_n^{1,n}$ there is a pure state $\psi$ of $\mathcal{A}$ extending $\varphi$,
see Theorem 5.1.13 in \cite{Murphy}, thus all pure states of $ \mathcal{D}_n^{1,n}$ have the form given in (\ref{gen-pure-states}).
For $x_0\in (0,+\infty)$, the evaluation map $\pi_{x_0}: \mathcal{D}_n^{1,n} \rightarrow M_n(\mathbb{C})=B(\mathbb{C}^n)$ 
is an irreducible representation because it is surjective. 
Thus we have the pure states of $\mathcal{D}_n^{1,n}$ given by
\begin{equation}\label{estado}
	f_{x_0,v}(M)=\langle M(x_0)v,v\rangle, \quad M \in \mathcal{D}_n^{1,n},
\end{equation}
where  $v \in \mathbb{C}^n$ is a unimodular  vector. 
In addition, for $x_1\in (0,+\infty)$ with $x_0\neq x_1$, take a function $f\in C[0,\infty]$ such that
$f(0)=f(\infty)=f(x_1)=1$ and $f(x_0)=0$. Since $f \mathcal{D}_n^{1,n} \subset \mathcal{D}_n^{1,n}$, the representation
$\pi_{x_1}$ is not unitarily equivalent to $\pi_{x_0}$.
In each case $x_0=0,\infty$, the representation  $\pi_{x_0}: \mathcal{D}_n^{1,n} \rightarrow M_n(\mathbb{C})=B(\mathbb{C}^n)$ 
is  completely reducible, actually $\pi_{x_0}$ is unitarily equivalent to the direct sum $\pi_{x_0,1} \oplus \cdots\oplus \pi_{x_0,n}$,
where $\pi_{x_0,j}$ is the one-dimensional representation of $\mathcal{D}_n^{1,n}$ given by
$\pi_{x_0,j}(M)=M_{jj}(x_0)$. These one-dimensional representations are pure states of $ \mathcal{D}_n^{1,n} $,
and they can be rewritten as
$$f_{x_0,e_j}(M)=\langle M(x_0)e_j,e_j \rangle,$$
where  $\{e_1,\dots, e_n\}$ is the canonical basis for $\mathbb{C}^n$. 
Moreover, each $M\in \mathcal{D}_n^{1,n}$  satisfies that  $M_{jj}(0)=M_{jj}(\infty)$ for all $ j=2,\dots,n-1 $, then the pure states $f_{0,e_j}$ and $f_{\infty,e_j}$ coincide.  There are no more pure states of $\mathcal{D}_n^{1,n}$.

\begin{theorem}\label{isPT}
	The $C^*$-algebra $\mathcal{C}^*(P_T; \{Q_{\Pi,(j)}\}_{j=1}^n)$ is isomorphic and isometric to the $C^*$-algebra $\mathcal{D}_n^{1,n}$. 
	Equivalently, the $C^*$-algebra 
	$\mathcal{D}_n^{1,n}$ is generated by 
	\begin{equation*}\label{Pgamma-Qjs}
		P_\gamma, Q_1,\dots, Q_{n}.
	\end{equation*}
	Moreover, the map $\mathcal{C}^*(P_T; \{Q_{\Pi,(j)}\}_{j=1}^n) \ni S\longmapsto \mathbf{R}  S \mathbf{R}^* \in \mathcal{D}_n^{1,n}$ is an isometric isomorphism of $C^*$-algebras, where
	\begin{equation*}
		P_T \longmapsto P_\gamma \quad \text{and} \quad Q_{\Pi,(j)} \longmapsto Q_j.
	\end{equation*}
\end{theorem}

\begin{proof}   
	Lemmas \ref{lem:sep-x0-x1-2}, \ref{lem:sep-ext-2} and \ref{lem:sep-extyx0-2} show that $\mathcal{B}^{1,n}$ separates all the pure states of $\mathcal{D}_n^{1,n}$. Then, the noncommutative  Stone-Weierstrass theorem  implies that 
	$\mathcal{D}_n^{1,n}=\mathcal{B}^{1,n}$.
\end{proof} 

\vspace{0.3cm}
We proceed to prove  that the algebra $\mathcal{B}^{1,n}$ separates all pure states of $\mathcal{D}_n^{1,n}$. 
The following lemma states that
the $C^*$-algebra $\mathcal{B}^{1,n}$ separates the pure states $f_{x_0,v}$ and $f_{x_1,w}$ when $x_0,\, x_1 \in (0,\infty)$ are distinct real numbers.

\begin{lemma}\label{lem:sep-x0-x1-2}
	Let $v, w \in \mathbb{C}^n$ be unit vectors,  and $x_0, x_1 \in (0, \infty)$. 
	Suppose that
	\begin{align}
		f_{x_0,v}(Q_j P_\gamma I Q_k)=f_{x_1,w}(Q_jP_\gamma I Q_k), \label{fQjQk-2}\\
		f_{x_0,v}(Q_j P_\gamma I Q_kP_\gamma I Q_j)=f_{x_1,w}(Q_jP_\gamma I Q_kP_\gamma I Q_j) \label{fQjQkQl-2}
	\end{align}
	for all $j,k=1,\dots,n$.
	Then $x_0=x_1$ and $v=\lambda w$, where $\lambda$ is a unimodular complex number.
\end{lemma}

\begin{proof}    
	From (\ref{fQjQk-2}) and (\ref{fQjQkQl-2}) we get
	\begin{align*}
		\langle Q_j(x_0) P_\gamma(x_0)Q_k(x_0) v,v \rangle &= \langle Q_j(x_1) P_\gamma(x_1)Q_k(x_1) w,w \rangle,\\
		\langle Q_j(x_0) P_\gamma(x_0)Q_k(x_0)P_\gamma(x_0)Q_j(x_0)v,v \rangle &= \langle Q_j(x_1) P_\gamma(x_1)Q_k(x_1)P_\gamma(x_1)Q_j(x_1) w,w \rangle.
	\end{align*}
	That is, 
	\begin{align}
		\sqrt{\gamma^{a_j}(x_0)\gamma^{a_k}(x_0)}v_k\overline{v_j}  
		&= \sqrt{\gamma^{a_j}(x_1)\gamma^{a_k}(x_1)}w_k\overline{w_j},
		\label{ljk} \\
		\gamma^{a_j}(x_0)\gamma^{a_k}(x_0) |v_j|^2 
		&=
		\gamma^{a_j}(x_1)\gamma^{a_k}(x_1)|w_j|^2. \label{ljl}
	\end{align}
	
	In the case $j=k$, Equation $\ref{ljk}$ reduces to
	\begin{align}
		\gamma^{a_j}(x_0)|v_j |^2  
		&= \gamma^{a_j}(x_1)|w_j|^2.
		\label{ljk1}
	\end{align}

	For $x_0, x_1 \in (0, \infty)$, the values $\gamma^{a_j}(x_0)$ and $\gamma^{a_j}(x_1)$ are different from zero for all $j=1,\dots,n$. Then \eqref{ljk1} implies that $v_j\neq 0$ if and only if $w_j\neq 0$.
	Choose $j$ such that $v_j \neq 0$. 
	Substituting (\ref{ljk1}) in (\ref{ljl}) and reducing, we obtain
	\begin{equation*}\label{gak}
		\gamma^{a_k}(x_0) = \gamma^{a_k}(x_1), \quad \text{ for all }\, k=1,\dots,n.   
	\end{equation*}
	In particular $\gamma^{a_n}(x_0)=\gamma^{a_n}(x_1)$. That is, 
	\begin{align*}
		\int_{2(n-1)x_0}^{\infty}(\ell_0(y))^2 dy&=\int_{2(n-1)x_1}^{\infty}(\ell_0(y))^2dy, 
	\end{align*}
	which implies $x_0=x_1$. 
	From (\ref{ljk}) we get
	$v_k\overline{v_j}=w_k\overline{w_j}\ $ for all $ k=1,\dots,n.$ 
	Thus
	\begin{equation*}\label{razon1}
		\frac{v_k}{w_k}=\frac{\overline{w_j}}{\overline{v_j}},  
	\end{equation*}
	for all $k$ with $v_k\neq 0$. At this point, define $\lambda=\overline{w_j}/\overline{v_j}$, note that $\lambda$ does not depend on $j$.
	Then we can write 
	$v_k=\lambda w_k$  for all  $k=1,\dots,n$, or equivalently
	$v=\lambda w$, where $|\lambda|=1.$
\end{proof}

The following lemma shows that the algebra $\mathcal{B}^{1,n}$ separates the pure states 
$f_{x_0,v}$ and $f_{x_1,e_k}$, where $x_0\in (0, \infty)$, $x_1\in\{0,\infty\}$ and $k\in \{1,\dots,n\}.$

\begin{lemma}\label{lem:sep-ext-2} 
	Let $v=(v_1,\dots,v_n)^t \in \mathbb{C}^n$ be a unit vector, 
	$x_0 \in (0,\infty)$ and $x_1=0, \infty$.
	If $v_j\neq 0$, then
	$$f_{x_0,v}(Q_j P_\gamma I Q_j) \neq f_{x_1,e_k}(Q_j P_\gamma I Q_j),\ \ k=1,\dots,n.$$
\end{lemma}

\begin{proof}  
	Since $P_\gamma(0)=E_{nn}$ and $P_\gamma(\infty)=E_{11}$, we have
	\begin{align*}
		f_{0,e_k}(Q_j P_\gamma I Q_j)&= \langle Q_j(0)P_\gamma(0)Q_j(0) e_k, e_k\rangle \\
		&= \langle E_{jj}E_{nn}E_{jj} e_k, e_k\rangle\\
		&=\delta_{jn} \delta_{kn}, \quad k=1,\dots,n. \\
	\end{align*}
	Similarly
	\begin{align*}
		f_{\infty,e_k}(Q_j P_\gamma I Q_j)&=\delta_{j1} \delta_{k1}, \quad k=1,\dots,n.
	\end{align*}
	Finally, we have
	$f_{x_0,v}(Q_j P_\gamma I Q_j)=\gamma^{a_j}(x_0) v_j \overline{v_j}$ with $\gamma^{a_l}(x_0)\in (0,1)$. 
	Suppose that $v_j\neq 0$, then
	$f_{x_0,v}(Q_j P_\gamma I Q_j) \in (0,1)$. 
	Hence,
	$f_{x_0,v}(Q_j P_\gamma I Q_j) \neq f_{x_1,e_k}(Q_j P_\gamma I Q_j)$ for  $k=1,\dots,n$.
\end{proof}

\begin{lemma}\label{lem:sep-extyx0-2} 
	The algebra $\mathcal{B}^{1,n}$ separates all the pure states 
	$$f_{0,e_1},\dots,f_{0,e_n}, f_{\infty,e_1},f_{\infty,e_n}$$
	of the algebra $\mathcal{D}_n^{1,n}$.
\end{lemma}

\begin{proof}  
	Since $f_{0,e_k}(Q_j)=\delta_{kj}$,
	then $Q_j$ separates the pure states $f_{0,e_j}$ and $f_{0,e_k}$ for $j\neq k$. 
	
	Note that  $f_{0,e_k}(Q_k)=1$ and $f_{\infty,e_1}(Q_k)=\delta_{1k}$,
	then $f_{0,e_k}$ and $f_{\infty, e_1}$ are separated by $Q_k$ for $k=2, \ldots, n.$ 
	Similarly, for $k=1, \dots, n-1$, the pure states $f_{0,e_k}$ and $f_{\infty, e_n}$ are separated by $Q_k$.
	
	Finally
	$$f_{0,e_1}(P_\gamma I)=\langle E_{nn} e_1, e_1\rangle =0, \quad f_{\infty,e_1}(P_\gamma I)=\langle E_{11} e_1, e_1\rangle=1,$$
	$$f_{0,e_n}(P_\gamma I)=\langle E_{nn} e_n, e_n\rangle=1, \quad f_{\infty,e_n}(P_\gamma I)=\langle E_{11} e_n, e_n\rangle=0.$$
	From this, $P_\gamma I$ separates the states $f_{0,e_1}$ and $f_{\infty,e_1}$,  it also separates the states  $f_{0,e_n}$ and $f_{\infty,e_n}$, 
	additionally it separates $f_{\infty,e_1}$ and $f_{\infty,e_n}$.
	The proof is complete.
\end{proof} 

\vspace*{.5cm}

We now turn to our main purpose:  on the existence of orthogonal projections 
and their relationship with $C^*$-algebras generated by Toeplitz operators acting
on poly-Bergman spaces.
Since $\mathcal{A}_{n}^2(\Pi)= \oplus_{k=1}^n \mathcal{A}_{(k)}^2(\Pi)$, and
$\mathcal{A}_{(k)}^2(\Pi)$
is isomorphic to the Bergman space $\mathcal{A}^2(\Pi)$,
the poly-Bergman space $\mathcal{A}_{n}^2(\Pi)$ is isomorphic to 
$(\mathcal{A}^2(\Pi))^n$.
Then, the system of
orthogonal projections $P_T,Q_{\Pi,(1)},\dots,Q_{\Pi,(n)}$
has a counterpart in a system of orthoprojections acting on $\mathcal{A}_{n}^2(\Pi)$.
We shift the scenario from the first part of this section to
the study of the $C^* $-algebra generated by the system of all-but-one orthogonal projections
$P,B_{\Pi,(1)}$,\dots,$B_{\Pi,(n)}$ in generic position,
where $B_{\Pi,(k)}$ is the orthogonal projection
from $\mathcal{A}_n^2(\Pi)$ onto $\mathcal{A}_{(k)}^2(\Pi)$, and
\begin{equation}\label{def-P}
	P:=  R_n^* \mathbf{R}  P_T \mathbf{R}^* R_n:\mathcal{A}_n^2(\Pi) \longrightarrow \mathcal{A}_n^2(\Pi).
\end{equation}
Section \ref{RK.} is devoted  to identifying the image of 
$P$ as a reproducing kernel Hilbert space. 
On the other hand, we  have
$$B_{\Pi,(j)}= R_n^* \mathbf{R}  Q_{\Pi,(j)} \mathbf{R}^* R_n,$$
where the operator $R_n$, given in (\ref{mapA_Ln}), identifies $\mathcal{A}_n^2(\Pi)$ with $(L_2(\mathbb{R}_+))^n$. 
Then the following commutative diagram shows the operators and spaces we are dealing with
$$
\xymatrix{  (\mathcal{A}^2(\Pi))^n \ar[d]_{\mathbf{R}} \ar[rr]^{P_T, \ Q_{\Pi,(j)}}  & &   (\mathcal{A}^2(\Pi))^n  \ar[d]^{\mathbf{R}} \\
	(L_2(\mathbb{R_+}))^n \ar[d]_{R_n^*} \ar[rr]^{P_\gamma I, \ Q_j}  & & (L_2(\mathbb{R_+}))^n \ar[d]^{R_n^*} \\
	\mathcal{A}_n^2(\Pi) \ar[rr]^{P, \ B_{\Pi,(j)}}  &  & \mathcal{A}_n^2(\Pi)  
}$$

\begin{theorem}\label{C-P-Bj}
	The $C^*$-algebra $\mathcal{C}^*(P;\{B_{\Pi,(j) }\}_{j=1}^n)$ is isomorphic and isometric to the $C^*$-algebra $\mathcal{D}_n^{1,n}$. 
	Moreover, the map $\mathcal{C}^*(P;\{B_{\Pi,(j) }\}_{j=1}^n) \ni T\longmapsto R_{n}  T R^*_{n} \in \mathcal{D}_n^{1,n}$ 
	is an isometric isomorphism of $C^*$-algebras, where
	\begin{equation*}
		P \longmapsto P_\gamma \quad \text{and} \quad B_{\Pi,(j)} \longmapsto Q_j.
	\end{equation*}
\end{theorem}

Recall that $\mathcal{D}_n^{\mathbb{C}I}$ is the $C^*$-subalgebra  of $\mathcal{D}_n^{1,n}$ consisting of all 
matrix-valued functions $M$ such that  $M(0)=M( \infty)$.

\begin{remark} According to Theorem \ref{CAToeplitz1}, the $C^*$-algebra $\mathcal{C}^*(\mathcal{T}_n^{0,0})$ is 
	isomorphic and isometric to $\mathcal{D}_n^{\mathbb{C}I}$, that is, 
	the $C^*$-algebra generated by all Toeplitz operators $T_{n,b}$, with symbols $b \in L_\infty^{\{0,\infty\}}(\mathbb{R}_+)$ 
	satisfying the condition $b^0=b^{+\infty}$, is a $C^*$-subalgebra of $\mathcal{C}^*(P;\{B_{\Pi,(j) }\}_{j=1}^n)$.
\end{remark}


\section{Reproducing kernels for the ranges of $P_T$ and $P$}\label{RK.}

We have seen that a $C^*$-algebra generated by Toeplitz operators acting on $\mathcal{A}_{n}^2(\Pi)$
can be obtained from a system of all-but-one orthogonal projections in generic position, where 
one of the orthoprojections is $P$, which plays a fundamental role in our study.
Now  $\mathcal{A}_{n}^2(\Pi)$ is a reproducing kernel Hilbert space; hence, all of its subspaces are as well.
In this section, we derive  integral representations of the reproducing kernel of the image of $P$,
and  exhibit their connection to some special functions.
The operator $P$ is constructed via $P_T$ and $P_{\gamma}$ by a unitary equivalence;
we also describe the images of these projections.
With this in mind, introduce the bounded continuous vector-valued function
\begin{equation}\label{def:Mn}
	M_n=(\sqrt{\gamma^{a_1}}, \dots, \sqrt{\gamma^{a_n}})^t: [0,\infty]\longrightarrow \mathbb{R}^n  .
\end{equation}
Note that $P_\gamma=M_n M_n^t$, and the Euclidean norm of $M_n(x)\in \mathbb{C}^n$ equals 1
for each $x\in[0,\infty]$. That is, $ M_n^tM_n=1$.

\begin{prop} 
	The following results hold:
	\begin{enumerate}
		\item The operator $P_{\gamma }$ is the orthogonal projection from $(L_2(\mathbb{R_+}))^n$ 
		onto the space of all functions of the form $g=aM_n$ with $a  \in L_2(\mathbb{R_+})$.
		\item The operator $P_{T }$ is the orthogonal projection from $(L_2(\Pi))^n$ onto the space of
		all vector-valued holomorphic functions of the form 
		\begin{align*}\label{intvec}
			h(z)&
			=\frac{1}{\sqrt{2\pi}}\int_{\mathbb{R_+}} e^{itz} \sqrt{2t} \, a(t) M_n(t)dt, \quad z=x+iy,
		\end{align*}
		where $a  \in L_2(\mathbb{R_+})$.
	\end{enumerate}
\end{prop}

\begin{proof}  
	To prove the first statement, let $f=(f_1, \dots, f_n)^t \in (L_2(\mathbb{R_+}))^n$. Thus,
	$$g:=P_\gamma f=  M_n (M_n^t f)=(M_n^t f)M_n=\widetilde{a} M_n,$$
	where $\widetilde{a}=M_n^t f=\sum_{k=1}^n \sqrt{\gamma^{a_k}} f_k$ belongs to $L_2(\mathbb{R_+})$. 
	On the other hand, if $a\in L_2(\mathbb{R_+})$, the function $g:=aM_n$ belongs to $(L_2(\mathbb{R_+}))^n$, and $$P_{\gamma}(g)=M_nM_n^taM_n=aM_n(M_n^tM_n)=aM_n=g.$$ 
	
	Thus,	$g\in \text{Im}\, P_\gamma$ if and only if there
	exists $a \in L_2(\mathbb{R_+})$ such that $g=aM_n$ almost everywhere.
	
	On the other hand, since $P_T=\mathbf{R}^*   P_{\gamma} \mathbf{R}$, 
	the image of $P_T$ consists of all functions of the form
	$h=\mathbf{R}^*g,$
	where $g=aM_n$ and $a  \in L_2(\mathbb{R_+})$. 
	Since $\mathbf{R}^*=\bigoplus_{k=1}^{n}R_1^*$, the second statement in the proposition  follows from  \eqref{Restrella}.
\end{proof}

\vspace*{.5cm}

\begin{theorem}\label{int-rep-PT}
	The orthogonal projection $P_T: (\mathcal{A}^2(\Pi))^n \rightarrow (\mathcal{A}^2(\Pi))^n$
	admits the integral representation
	$$(P_T \phi)(z)=\int_{\Pi} \phi(u,v)\, \mathbf{K} (z,w)  du dv, \quad w=u+iv,$$
	where
	$$\mathbf{K}(z,w)= \frac{1}{2\pi} \int_{\mathbb{R}_+}  e^{it(x-u)}\, 2t\, \ell_0(2ty) \, \ell_0(2tv) \, P_{\gamma}(t) dt,$$
	with $z=x+iy$ and $w=u+iv$.
\end{theorem}

\begin{proof}  
	Let $\mathbf{I}$ be the $n\times n$ identity matrix. 
	Then $\mathbf{R}= \bigoplus_{k=1}^{n} R_1=R_1 \mathbf{I}$. 
	According to (\ref{P-gamma}) we have $P_T=  \mathbf{R}^* P_\gamma \mathbf{R}$.
	Explicit formulae for the values of $\mathbf{R}$ and $\mathbf{R}^*$ follow from (\ref{Rn}) and (\ref{Restrella}), 
	they are given by
	\begin{equation*}
		(\mathbf{R} \phi)(t)= \frac{\chi_+(t)}{\sqrt{2\pi}} \int_{\mathbb{R}_+} 
		\int_{\mathbb{R}} e^{-itu}  \sqrt{2t}\, \ell_0(2tv)\, \phi(u,v)\, du dv, \quad \phi \in (\mathcal{A}^2(\Pi))^n
	\end{equation*}	
	and
	\begin{equation*}
		(\mathbf{R}^* f)(x,y)= \frac{1}{\sqrt{2\pi}} \int_{\mathbb{R}_+} e^{itx}
		\sqrt{2t}\, \ell_0(2ty)\, f(t) dt, \quad f \in (L_2(\mathbb{R_+}))^n.
	\end{equation*}	
	
	Take $\phi=(\phi_1, \dots,\phi_n)^t\in (\mathcal{A}^2(\Pi))^n$. Then, 
	\begin{align*}
		(P_T\phi)(x,y)
		&=\left(\left[ \mathbf{R}^* P_\gamma  \mathbf{R}\right]\phi\right)(x,y)\\
		&=\frac{1}{\sqrt{2\pi}} \int_{\mathbb{R}_+} e^{itx}
		\sqrt{2t}\, \ell_0(2ty)\, P_{\gamma}(t)\, (\mathbf{R}\phi)(t)\, dt \\
		&= \frac{1 }{2\pi}  \int_{\mathbb{R_+}} e^{itx}   \int_{\mathbb{R}_+}  \int_{\mathbb{R}}
		e^{-itu} \, 2t \,	\ell_0(2ty) \ell_0(2tv) P_{\gamma}(t)    \phi(u,v) dudvdt. 
	\end{align*}
	The integral representation of $\mathbf{K}(z,w)$  immediately follows from last integral.
\end{proof}

\normalsize

\vspace*{.5cm}
An interesting fact is that the image of $P_T$ involves the digamma function, which in turn, 
is related to the Harmonic numbers given by
$H_n=\sum\limits_{k=1}^{n} 1/k$, with $n\in \mathbb{N}.$
Recall that the digamma function is defined by 
$$\psi(z)=\frac{\Gamma'(z)}{\Gamma(z)}
=-\gamma+\sum_{n=0}^{\infty} \left( \frac{1}{n+1}-\frac{1}{n+z} \right), \ \ z\notin \{0, -1,-2,-3, \dots\} $$
\noindent
where $\Gamma $ is the Gamma function
and $\gamma=-\psi(1)$ is the Euler-Mascheroni constant.
The function $\psi$ satisfies $\psi(n)=H_{n-1}-\gamma$, where $H_0=0$.
Thus, Harmonic numbers are values of the harmonic function
$$H(z)=\psi(z+1)+\gamma.$$
Since $\psi(z+1)=\psi(z)+1/z$, one can easily  prove that $H(z-1)=H(z)-1/z$.
In addition, the Nielsen's beta function also emerges in this framework, it is defined by the formula
$$\beta(z)=\int_{0}^{\infty} \frac{e^{-zt}}{1+e^{-t}} dt 
= \sum_{n=0}^{\infty} \frac{(-1)^{n}}{z+n}, \ \ \text{Re} \ z>0.$$
This function satisfies
\begin{align*}
	\beta(z)
	&=\frac{1}{2}\left(\psi\left(\frac{z+1}{2}\right)-\psi\left(\frac{z}{2}\right)\right) \\
	&=\frac{1}{2}\left(H\left(\frac{z+1}{2}\right)-H\left(\frac{z}{2}\right)\right) + \frac{1}{z} -\frac{1}{z+1}.
\end{align*}

Some computations below require the following result
$$\int_{0}^{\infty} e^{-pt} (1-e^{-t/{\alpha}})^{\nu-1} dt=\alpha B(\alpha p,\nu), 
\quad \text{Re}\, \alpha, \ \text{Re}\, \nu, \ \text{Re}\, p >0,$$
where $B(z,w)$ is the usual beta function. In particular, we introduce the 
following function defined by the Laplace transform of the function $h(t)=(1-e^{-2t})^{1/2}$:
$$J(p):=\int_{0}^{\infty} e^{-pt} (1-e^{-2t})^{1/2} dt=\frac{1}{2} B(p/2,3/2)
=\frac{\sqrt{\pi}}{4} \frac{\Gamma(p/2)}{\Gamma\left(\frac{p}{2}+\frac{3}{2}\right)}.$$

\begin{lemma}
	The Laplace transform of function $th(t)$ is given by
	\begin{equation}\label{integral-Kjn}
		\int_{0}^{\infty}  e^{-pt} \, t (1-e^{-2t})^{1/2} dt =- \frac{d}{dp} J(p)
		=J(p) \left( \beta(p)+\frac{1}{p+1}\right).
	\end{equation}
\end{lemma}

\begin{proof}
	Compute
	\begin{align*}
		\int_{0}^{\infty} e^{-pt}\, t(1-e^{-2t})^{1/2} dt
		&= - \frac{d}{dp} J(p) \\
		&= -\frac{\sqrt{\pi}}{4} \frac{1}{2} 
		\frac{\Gamma\left(\frac{p}{2}+\frac{3}{2}\right) \Gamma'\left(\frac{p}{2}\right)-
			\Gamma\left(\frac{p}{2}\right) \Gamma'\left(\frac{p}{2}+\frac{3}{2}\right)}
		{\left(\Gamma\left(\frac{p}{2}+\frac{3}{2}\right)\right)^2} \\
		&= -\frac{\sqrt{\pi}}{8} 
		\frac{\Gamma\left(\frac{p}{2}\right) \psi\left(\frac{p}{2}\right) -
			\Gamma\left(\frac{p}{2}\right) \psi\left(\frac{p}{2}+\frac{3}{2}\right)}
		{\Gamma\left(\frac{p}{2}+\frac{3}{2}\right)} \\
		&= \frac{\sqrt{\pi}}{4} 
		\frac{\Gamma\left(\frac{p}{2}\right)}{\Gamma\left(\frac{p}{2}+\frac{3}{2}\right)} 
		\cdot \frac{1}{2}
		\left( \psi\left(\frac{p}{2}+\frac{3}{2}\right)-\psi\left(\frac{p}{2}\right)\right).
	\end{align*}
	Formula (\ref{integral-Kjn}) follows from the identity $\psi(z+1)=\psi(z)+1/z$.
\end{proof}

Introduce now the functions
\begin{equation*}\label{Kernel-m-z-w}
	G_m(z,w)=\frac{-1}{\pi\left[z-\overline{w}+m i\right]^2}, \quad m\in \mathbb{Z}.
\end{equation*}

\begin{theorem} 
	Let $z=x+iy$ and $w=u+iv$. The reproducing kernel $\mathbf{K}(z,w)$ of the image of $P_T$ is given by
	\begin{equation*}\label{kernerG}
		\mathbf{K}(z,w)= (K_{jk}(z,w))_{j,k=1,\dots,n},
	\end{equation*}
	where
	\begin{equation*}
		K_{nn}(z,w)=G_{2(n-1)}(z,w) = \frac{-1}{\pi\left[z-\overline{w}+2(n-1)i \right]^2},
	\end{equation*}
	\begin{equation*}
		K_{jk}(z,w)= -G_{j+k}(z,w)+G_{j+k-2}(z,w),
		\quad j, k=1,\dots,n-1,
	\end{equation*}
	and
	\begin{align*}
		K_{jn}(z,w)
		&= \frac{1}{\pi} J(p)  \left( \beta(p)+\frac{1}{p+1}\right) \\
		&= \frac{J(p)}{2\pi}\left[ H\left(\frac{p+1}{2}\right)-H\left(\frac{p}{2}\right) + \frac{2}{p} \right]
	\end{align*}
	where $j=1, \dots, n-1$ and $p=-i(z-\overline{w})+j+(n-2)$.
\end{theorem}

\begin{proof}   From Theorem \ref{int-rep-PT}, the reproducing kernel of the image of $P_T$
	equals
	\begin{equation*}
		\mathbf{K}(z,w)= (K_{jk}(z,w))_{j,k=1,\dots,n},
	\end{equation*}
	where
	$$K_{jk}(z,w)= \frac{1}{2\pi}\int_{\mathbb{R}_+} e^{it(x-u)} 
	2t  \ell_{0}(2ty)\ell_{0}(2tv) \sqrt{\gamma^{a_j}(t)\gamma^{a_k}(t)} dt,$$
	and the functions $\gamma^{a_k}$ are defined in (\ref{eq.gk}) and (\ref{eq:gn}). 
	Recall that $\ell_{0}(t)=\chi_+(t) e^{-t/2}$. 
	For $j=k=n$, we have that
	\begin{align*}
		K_{nn}(z,w)
		&= \frac{1}{\pi} \int_{\mathbb{R}_+}t e^{it(x-u)-ty-tv} e^{-2(n-1)t} \, dt \\
		&= \frac{1}{\pi} \int_{\mathbb{R}_+}t e^{it[z-\overline{w}+2(n-1)i]} \, dt, \quad {\text{Im}} (z-\overline{w}) >0\\
		&= \frac{-1}{\pi\left[z-\overline{w}+2(n-1)i \right]^2}.
	\end{align*}
	On the other hand, for $j,k=1,\ldots,n-1$, we have 
	\begin{align*}
		K_{jk}(z,w)
		&= \frac{1}{\pi}\int_{\mathbb{R}_+}t e^{it(x-u)-ty-tv} \sqrt{e^{-2jt}[-1+e^{2t}] e^{-2kt}[-1+e^{2t}] } dt\\
		&=\frac{1}{\pi}\int_{\mathbb{R}_+}t e^{it[z-\overline{w}+ji+ki]} [-1+e^{2t}] dt\\
		&=\frac{-1}{\pi}\int_{\mathbb{R}_+}t e^{it[z-\overline{w}+ji+ki]}dt +
		\frac{1}{\pi}\int_{\mathbb{R}_+}t e^{it[z-\overline{w}+ji+ki-2i]} dt.  \\
		&=\frac{1}{\pi[z-\overline{w} +(j+k)i]^2} - \frac{1}{\pi[z-\overline{w} +(j+k-2)i]^2}. 
	\end{align*} 
	
	Finally, let $j\in \{1,\dots,n-1\}$ and introduce $q=i(u-x)+y+v+j+n-1$.
	Note that $\text{Re}\, q > 2$. Then
	
	\begin{align*}
		K_{jn}(z,w)
		&= \frac{1}{\pi} \int_{\mathbb{R}_+} te^{it(x-u)-ty-tv}  \sqrt{\gamma^{a_j}(t)\gamma^{a_n}(t)} \, dt \\
		&= \frac{1}{\pi} \int_{\mathbb{R}_+} t e^{-t[i(u-x)+y+v+j+n-1]} \sqrt{-1+e^{2t}} dt \\
		&= \frac{1}{\pi} \int_{\mathbb{R}_+} t  e^{-pt}  \sqrt{1-e^{-2t}} dt, \quad p=q-1,
	\end{align*}
	where $p=-i(z-\overline{w})+j+(n-2)$. Formula (\ref{integral-Kjn}) finishes the proof.
\end{proof}

\vspace{.6cm}
In order to have an integral representation of $P$, recall that the image of $P_\gamma $ consists of 
all functions of the form $g=aM_n$, where $a\in L_2(\mathbb{R_+})$ and $M_n$ is given in (\ref{def:Mn}). 
Thus, the image of $P$ consists of all functions $h=R_n^*(aM_n)$, where $a \in L_2(\mathbb{R_+})$, that is, 
$$h(z)=\frac{1}{\sqrt{2\pi}}\int_{\mathbb{R}_+} e^{i xt} \sqrt{2t}\,  a(t)[N_n(2ty)]^t M_n(t)dt.$$

\begin{theorem}\label{Int-rep-P}
	The orthogonal projection $P:\mathcal{A}_n^2(\Pi) \rightarrow \mathcal{A}_n^2(\Pi)$
	admits the integral representation
	$$(P\phi)(z)=\int_{\Pi} \phi(u,v)\, K^{\gamma} (z,w) du dv, \quad w=u+iv,$$
	where
	\begin{equation}\label{Kgamma}
		K^{\gamma}(z,w)= \frac{1}{\pi} \int_{\mathbb{R}_+}  e^{it(x-u)} 
		t \, N_n(2ty)^t P_{\gamma}(t)  N_n(2tv) dt, \quad z=x+iy.
	\end{equation}
\end{theorem}

\begin{proof}  
	We have $P=  R_n^* P_\gamma R_n$. Take $\phi \in \mathcal{A}_n^2(\Pi)$,
	then
	\begin{equation*}
		(R_{n} \phi)(t)= \frac{\chi_+(t)}{\sqrt{2\pi}} \int_{\mathbb{R}_+} 
		\int_{\mathbb{R}} e^{-itu} \phi(u,v) \sqrt{2t} N_n(2tv) du dv, 
	\end{equation*}	
	and
	\begin{align*}
		(P\phi)(x,y)
		&= (R_n^* P_\gamma R_n \phi)(x,y)  \\
		&= \frac{1 }{\sqrt{2\pi}}\int_{\mathbb{R}_+} e^{itx} \sqrt{2t} N_n(2ty)^t P_{\gamma}(t) (R_n\phi)(t) dt\\
		&=\frac{1 }{2\pi}\int_{\mathbb{R}_+} e^{itx} 2t \int_{\mathbb{R}_+}  
		\int_{\mathbb{R}} e^{-itu} \phi(u,v) N_n(2ty)^t P_{\gamma}(t)  N_n(2tv) du dv dt\\
		&=\int_{\mathbb{R}_+}\int_{\mathbb{R}} \phi(u,v) \frac{1}{2\pi} 
		\int_{\mathbb{R}_+}  e^{it(x-u)} 2t N_n(2ty)^t P_{\gamma}(t)  N_n(2tv)dt du dv \\
		&=\int_{\mathbb{R}_+}\int_{\mathbb{R}} \phi(u,v)K^{\gamma}(z,w)\, du dv.
	\end{align*}
\end{proof}

\vspace*{.5cm}
It is easy to see that
$K^{\gamma}(z,w)=\sum_{j=1}^n\sum_{k=1}^n K^\gamma_{jk}(z,w)$,
where 
\begin{align*}
	K^{\gamma}_{jk}(z,w)
	&=\frac{1}{\pi}\int_{\mathbb{R}_+} e^{it(x-u)} t \sqrt{\gamma^{a_j}(t)\gamma^{a_k}(t)} \ell_{j-1}(2ty) \ell_{k-1}(2tv)dt. 
\end{align*}
Even though this integral representation of the reproducing kernel seems to be fairly simple,
a cumbersome computation of the integral in terms of the coefficients of Laguerre polynomials 
reveals no substantial information of the reproducing kernel. 
However, $P$ is unitarily equivalent to $P_T$, thus  the reproducing kernel  $K(z,w)$ theoretically contains all the information encoded in $K^{\gamma}(z,w)$.
We proceed to compute  another integral representation of $K^{\gamma}$.
Introduce the functions
$$\phi_m(t,y)= \sum_{j=1}^{m}  e^{-(j-1)t} \ell_{j-1}(2ty),$$
where $\ell_k$ is given in \eqref{elek} for $k=0,1,2,\dots$.

\begin{theorem}\label{K-gamma-2}
	The reproducing kernel $K^{\gamma}(z,w)$ is given by
	\begin{equation*}
		K^{\gamma}(z,w)=\frac{1}{\pi} \int_{\mathbb{R}_+} e^{it(x-u)} \, t \, 
		(\phi_n(t,y), \, \phi_{n-1}(t,y)) 
		\begin{pmatrix}
			1 & \phi(t) \\
			\phi(t) & [\phi(t)]^2
		\end{pmatrix}
		\begin{pmatrix}
			\phi_n(t,v) \\  \phi_{n-1}(t,v) 
		\end{pmatrix} dt,
	\end{equation*}
	where $\phi(t)=-1+\sqrt{1-e^{-2t}}$, $z=x+yi$ and $w=u+vi$.
	%
\end{theorem}

\begin{proof}  
	We write $M_n=(M_n'^t,\sqrt{\gamma^{a_n}})^t$ and $N_n=(N_{n-1}^t, \ell_{n-1})^t$, where
	$M_{n}'=(\sqrt{\gamma^{a_1}}, \dots, \sqrt{\gamma^{a_{n-1}}})^t$.
	Thus, by (\ref{eq.gk}),
	$$N_{n-1}(2ty)^t M_{n}'(t)= \sqrt{1-e^{-2t}} \phi_{n-1}(t,y).$$
	According to (\ref{Kgamma}), we calculate
	\begin{align*}
		k&:=N_n(2ty)^t P_{\gamma}(t)  N_n(2tv) \\
		& = [N_n(2ty)^t M_n(t)]\, [N_n(2tv)^t M_n(t)] \\
		&= [N_{n-1}(2ty)^t M_{n}'(t) +  \ell_{n-1}(2ty) \sqrt{\gamma^{a_n}(t)} ] \cdot \\
		& \ \ \ \ [N_{n-1}(2tv)^t M_{n}'(t) +  \ell_{n-1}(2tv) \sqrt{\gamma^{a_n}(t)} ] \\
		&= (1-e^{-2t}) \phi_{n-1}(t,y) \phi_{n-1}(t,v) +   \ell_{n-1}(2ty)   \ell_{n-1}(2tv)  \gamma^{a_n}(t)+ \\
		& \ \ \ \  \sqrt{1-e^{-2t}}  \sqrt{\gamma^{a_n}(t)} [ \phi_{n-1}(t,y)  \ell_{n-1}(2tv) +  \phi_{n-1}(t,v)  \ell_{n-1}(2ty) ].
	\end{align*}
	Recall that $\gamma^{a_n}(t)=e^{-2(n-1)t}$, moreover $
	e^{-(n-1)t}\ell_{n-1}(2ty)=\phi_n(t,y)-\phi_{n-1}(t,y)$.
	Then
	\begin{align*}
		k
		= & (1-e^{-2t}) \,  \phi_{n-1}(t,y)\phi_{n-1}(t,v) + \\
		& [\phi_n(t,y)-\phi_{n-1}(t,y)] [\phi_n(t,v)-\phi_{n-1}(t,v)] + \\
		&    \sqrt{1-e^{-2t})}  \,  [ \phi_{n-1}(t,y) \{ \phi_n(t,v)-\phi_{n-1}(t,v) \} + \phi_{n-1}(t,v) \{ \phi_n(t,y)-\phi_{n-1}(t,y) \} ].
	\end{align*}
	The  last form of $k$ is a quadratic expression with respect to 
	$\phi_n(t,y)$, $\phi_{n-1}(t,y)$,  $\phi_n(t,v)$ and $\phi_{n-1}(t,v)$.
	One can  easily see that
	$$N_n(2ty)^t P_{\gamma}(t)  N_n(2tv)=
	(\phi_n(t,y), \, \phi_{n-1}(t,y)) 
	\begin{pmatrix}
		1 & \phi(t) \\
		\phi(t) & [\phi(t)]^2
	\end{pmatrix}
	\begin{pmatrix}
		\phi_n(t,v) \\  \phi_{n-1}(t,v) 
	\end{pmatrix},$$
	where $\phi(t)=-1+\sqrt{1-e^{-2t}}$, for which $[\phi(t)]^2=2-2\sqrt{1-e^{-2t}}-e^{-2t}$.
	Finally, (\ref{Kgamma}) completes the proof.
\end{proof}

\vspace{.5cm}
All reproducing kernels
$K_{nn}^{\gamma}$ and $K_{jk}^{\gamma}$, with $1 \leq j,k \leq n-1$, can be expressed in terms of 
functions of the form 
$$(2yi)^j (2vi)^k \int_{\mathbb{R}_+} e^{it(z-\overline{w})-mt} t^{j+k+1} dt 
= (j+k+1)! \frac{(z-\overline{z})^j (w-\overline{w})^k}{[-i(z-\overline{w})+m]^{j+k+2}}.$$
Of course, this kind of integral appears in Theorem \ref{K-gamma-2}.
However, integrals corresponding to the reproducing kernels $K_{jn}^{\gamma}$ and $K_{nj}^{\gamma}$, 
with $1\leq j\leq n-1$, are given in terms of integrals of the form
\begin{equation*}
	\int_{0}^{\infty} t^m e^{-pt} (1-e^{-2t})^{1/2} dt =(-1)^m \frac{d^m}{dp^m} J(p).
\end{equation*}
This integral involves some special functions, and it suggests a further research for the reproducing kernel $K^{\gamma}(z,w)$.


\section{Algebra generated by a Toeplitz operator and $n$ orthogonal projections.}\label{sect_purespa}

In this section we study the $C^*$-algebra  $\mathcal{C}^*(T_{n,a_0},B_{\Pi,(1)},\dots,B_{\Pi,(n)})$ 
generated by the orthogonal projections $B_{\Pi,(1)},...,B_{\Pi,(n)}$ and
a Toeplitz operator $T_{n,a_0}: \mathcal{A}_n^2(\Pi) \rightarrow \mathcal{A}_n^2(\Pi)$ 
with vertical symbol $a_0$ to be defined below. 
This $C^*$-algebra is closely related to both the $C^*$-algebra $\mathcal{C}^*(P_T,B_{\Pi,(1)},\dots,B_{\Pi,(n)})$, and the $C^*$-algebra generated by Toeplitz operators acting on the poly-Bergman space, as we see at the end of this section.
We apply technique given in the Subsection \ref{subsect-alg-PT} to the study of the 
algebra under consideration.

Recall the orthogonal projection $Q_j $ given in (\ref{Qj}). This orthoprojection is unitarily equivalent to $B_{\Pi,(j)}$. We saw that
$$Q_j=R_n B_{\Pi,(j)} R_n^{*},$$
where $R_n$ is given in (\ref{mapA_Ln}).
Let $\alpha>0$ and 
$$a_0(z)=\chi_{[0,\alpha/2]}(y).$$ 
According to (\ref{gamma-n-a}), the  corresponding Toeplitz operator $T_{n,a_0}$ is unitarily equivalent to the operator $\gamma^{n,a_0}I$, acting on $(L_2(\mathbb{R_+}))^n$, where
\begin{equation*}
	\gamma^{n,a_0}(x)= \int_{0}^{\alpha x} N_n(y)[N_n(y)]^t dy,
\end{equation*}
and $N_n(y)=(\ell_0(y),\dots, \ell_{n-1}(y))^t$.

Let us consider the following commutative diagram:
$$
\xymatrix{ 
	\mathcal{A}_n^2(\Pi) \ar[d]_{R_n} \ar[rr]^{T_{n,a_0}, \ B_{\Pi,(j)}}  & &  \mathcal{A}_n^2(\Pi) \ar[d]^{R_n} \\ 
	(L_2(\mathbb{R}_+))^n \ar[rr]_{\gamma^{n,a_0}I,\ Q_j}  & & (L_2(\mathbb{R}_+))^n
}
$$

Recall that $Q_1,\dots, Q_n$ are identified with multiplication operators 
by matrix-valued functions defined on the positive real line. 
From this perspective, introduce the $C^*$-algebra
$$\mathcal{B}_n:=\mathcal{C}^*(\gamma^{n,a_0}, Q_1, \dots ,Q_{n}).$$
We will prove that $\mathcal{B}_n$ equals to the following  $C^*$-algebra
$$\mathcal{D}_n=\{ M\in M_n(\mathbb{C})\otimes C[0,\infty]\ :\ M(0),M(\infty) \text{ are diagonal} \}.$$

Each generator of $\mathcal{B}_n$ evaluated at the points $0, \infty$ is a diagonal matrix,
thus $\mathcal{B}_n \subset \mathcal{D}_n$. In order to prove that these $C^*$-algebras coincide, we will prove that $\mathcal{B}_n$ separates all the pure states of  $\mathcal{D}_n$. 
For $x_0 \in (0, \infty)$, we have the following pure states of $\mathcal{D}_n$:
$$f_{x_0,v}(M)=\langle M(x_0)v,v\rangle, \quad M \in \mathcal{D}_n,$$
with  $v \in \mathbb{C}^n$ a unit vector. 
For the boundary points $x_0=0,+\infty$, the pure states of $\mathcal{D}_n$ are given by
$$f_{x_0,e_j}(M)=\langle M(x_0)e_j,e_j \rangle, \quad  j=1,\dots,n,$$
where $\{e_1,\dots e_n\}$ is the canonical basis of $\mathbb{C}^n$.
There are no additional pure states of $\mathcal{D}_n$.

\begin{theorem}\label{C-Ta,B(j)}
	The $C^*$-algebra $\mathcal{C}^*(T_{n,a_0},B_{\Pi,(1)},\dots,B_{\Pi,(n)})$ 
	is isomorphic and isometric to the $C^*$-algebra $\mathcal{D}_n$.
	Equivalently, the $C^*$-algebra 
	$\mathcal{D}_n$ is generated by 
	\begin{equation*}\label{gamma-Qjs}
		\gamma^{n,a_0}, Q_1, \dots, Q_{n}.
	\end{equation*}
	Moreover, the map 
	$\mathcal{C}^*(T_{n,a_0},B_{\Pi,(1)},\dots,B_{\Pi,(n)}) \ni T\longmapsto R_n  TR_n^* \in \mathcal{D}_n$ 
	is an isometric isomorphism of $C^*$-algebras, where
	\begin{equation*}
		T_{n,a_0}\longmapsto \gamma^{n,a_0}I \quad \text{and} \quad B_{\Pi,(j)} \longmapsto Q_j.
	\end{equation*}
\end{theorem}

\begin{proof} Of course $\mathcal{D}_n$ is a liminal $C^*$-algebra. 
	The $C^*$-algebra $\mathcal{B}_n$  separates all the pure states of $\mathcal{D}_n$ as shown in 
	Lemmas \ref{lem:sep-x0-x1}, \ref{lem:sep-ext} and \ref{lem:sep-extyx0} below. 
	By \textit{the noncommutative Stone-Weierstrass theorem}, 
	we have that $\mathcal{D}_n=\mathcal{B}_n$.
\end{proof}

\begin{lemma}\label{lem:sep-x0-x1}
	Let $v, w \in \mathbb{C}^n$ be unit vectors,  and $x_0, x_1 \in (0, \infty)$. 
	Let $\gamma^{n,a_0}$ be the spectral function of the Toeplitz operator $T_{n,a_0}$, where $a_0(z)=\chi_{[0,\alpha/2]}(y)$.
	Suppose that
	\begin{align}
		f_{x_0,v}(Q_j\gamma^{n,a_0}I Q_k)& =f_{x_1,w}(Q_j\gamma^{n,a_0}I Q_k), \label{fQjQk}   \\
		f_{x_0,v}(Q_j\gamma^{n,a_0}I Q_k\gamma^{n,a_0}I Q_l)&=f_{x_1,w}(Q_j\gamma^{n,a_0}I Q_k\gamma^{n,a_0}I Q_l), \label{fQjQkQl}
	\end{align}
	for all $j,k,l=1,...,n.$
	Then $x_0=x_1$ and $v=\lambda w$, where $\lambda$ is a unimodular complex number.
\end{lemma}

\begin{proof}
	For the sake of simplicity, $a$ denotes the symbol $a_0$ along the proof.
	Elementary calculations show that $Q_j\gamma^{n,a} I Q_k=\gamma_{jk}^{n,a}E_{jk}$ and $E_{jk}E_{kl}=E_{jl}$.
	From (\ref{fQjQk}) and (\ref{fQjQkQl}) we obtain
	\begin{align*}
		\gamma_{jk}^{n,a}(x_0)v_k\overline{v_j}&=\gamma_{jk}^{n,a}(x_1)w_k\overline{w_j}, \label{gammaTP1} \\
		\gamma_{jk}^{n,a}(x_0)\gamma_{kl}^{n,a}(x_0)v_l\overline{v_j}
		&=\gamma_{jk}^{n,a}(x_1)\gamma_{kl}^{n,a}(x_1)w_l\overline{w_j} 
	\end{align*}
	for all $j,k,l=1,\dots,n.$ 
	Note that
	\begin{align*}
		\gamma_{jk}^{n,a}(x)&=\int_{0}^{\alpha x}\ell_{j-1}(y)\ell_{k-1}(y)dy \neq 0 \ \ \forall x\in (0,\infty).
	\end{align*}
	Thus, $v_k \overline{v_j}\neq 0$ if and only if $w_k \overline{w_j}\neq 0$.
	In particular, $v_j\neq 0$ if and only if $w_j\neq 0$.
	Therefore
	\begin{equation}\label{razon}
		\frac{\gamma_{jk}^{n,a}(x_0)}{\gamma_{jk}^{n,a}(x_1)}=\frac{w_k\overline{w_j}}{v_k\overline{v_j}}, 
	\end{equation}
	
	\begin{equation}\label{razon2}
		\frac{\gamma_{jk}^{n,a}(x_0)}{\gamma_{jk}^{n,a}(x_1)}\cdot
		\frac{ \gamma_{kl}^{n,a}(x_0)}{\gamma_{kl}^{n,a}(x_1)}
		=\frac{w_l\overline{w_j}}{v_l\overline{v_j}}
	\end{equation}
	whenever $v_k\overline{v_j}\neq 0$ and  $v_l\overline{v_j}\neq 0$. Since $v \neq 0$, there exists $k \in \{1,\dots,n\}$ such that $v_k\neq 0$.
	Then, using (\ref{razon}) and (\ref{razon2}) with $j=l=k$ we get that
	\begin{equation*}
		\frac{w_k\overline{w_k}}{v_k\overline{v_k}}
		\frac{w_k\overline{w_k}}{v_k\overline{v_k}}
		=\frac{w_k\overline{w_k}}{v_k\overline{v_k}}
	\end{equation*}
	or $|w_k|^2=|v_k|^2.$
	Using last result and (\ref{razon}) with $j=k$ we obtain $\gamma_{kk}^{n,a}(x_0)=\gamma_{kk}^{n,a}(x_1)$. That is,
	$$\int_0^{x_0 \alpha} (\ell_{k-1}(y))^2dy=\int_0^{x_1 \alpha} (\ell_{k-1}(y))^2dy,$$
	which implies that $x_0=x_1$.
	Consequently, from (\ref{razon}) we find that
	$v_k\overline{v_j}=w_k\overline{w_j}$
	for all $j$. Thus, if $v_j\neq 0$, then
	$$\frac{v_k}{w_k}=\frac{\overline{w_j}}{\overline{v_j}}.$$
	As in the proof of Lemma \ref{lem:sep-x0-x1-2}, $\lambda:= \frac{\overline{w}_j}{\overline{v}_j}$ does not depend on $j$. As a consequence,	$v=\lambda w$ with $|\lambda|=1.$
\end{proof}

\vspace*{.6cm}
We proceed to separate the rest of the pure states of  $\mathcal{D}_n$.

\begin{lemma}\label{lem:sep-ext}  
	Let $\gamma^{n,a_0}$ be the spectral function of the Toeplitz operator $T_{n,a_0}$, where $a_0(z)=\chi_{[0,\alpha/2]}(y)$. 
	Then 
	\begin{equation*}
		f_{0,e_j} (\gamma^{n,a_0}I) \neq f_{\infty,e_k}(\gamma^{n,a_0}I) \quad \text{ for all} \quad j,k=1,\dots,n.
	\end{equation*}
	In addition, the pure states $f_{x,e_1},\dots,f_{x,e_n}$ are separated by $Q_1,\dots,Q_n$, where $x=0, +\infty$.
\end{lemma}

\begin{proof}   
	Recall that  $\gamma^{n,a_0}(0)=0I$ and  $\gamma^{n,a_0}(\infty)=I$. Then
	$$f_{0,e_j}(\gamma^{n,a_0}I)=0 \quad
	\text{and} \quad
	f_{\infty,e_k}(\gamma^{n,a_0}I)=1.$$
	That is, the pure states $f_{0,e_j}$ and $f_{\infty,e_k}$ are separated by $\mathcal{B}_n$ for $j,k \in \{1,\ldots, n\}$.
	
	Let $x=0, +\infty$. Note that $f_{x,e_k}(Q_j)=\delta_{jk}$, thus $f_{x,e_j}(Q_j)\neq f_{x,e_k}(Q_j)$ for $j\neq k$.
\end{proof}

\begin{lemma}\label{lem:sep-extyx0} 
	Let $v=(v_1,\dots,v_n)^t \in \mathbb{C}^n$ be a unit vector,  $x_0 \in (0,\infty)$ and $x_1\in \{0, +\infty\}$ .
	Let $\gamma^{n,a_0}$ be the spectral function of the Toeplitz operator $T_{n,a_0}$, where $a_0(z)=\chi_{[0,\alpha/2]}(y)$. 
	If $v_j\neq 0$, then
	\begin{equation*}
		f_{x_0,v} (Q_j\gamma^{n,a_0}IQ_j) \neq f_{x_1,e_k}(Q_j\gamma^{n,a_0}IQ_j)
		\quad \text{ for all} \quad k=1,\dots,n.
	\end{equation*}
\end{lemma}

\begin{proof}  
	Note that $f_{0,e_k}(Q_j\gamma^{n,a_0}IQ_j)=0$ and 
	$f_{\infty,e_k}(Q_j\gamma^{n,a_0}I Q_j)=\delta_{jk}.$
	Choose $j$ such that $v_j \neq 0$.  Since $\gamma_{jj}^{n,a_0}(x_0) \in (0,1)$,
	\begin{align*}
		f_{x_0,v}(Q_j\gamma^{n,a_0}I Q_j)&=\gamma_{jj}^{n,a_0}(x_0)|v_j|^2 \in (0,1),
	\end{align*}
	which means that $f_{x_0,v}$ and $f_{x_1,e_k}$ are separated by $\mathcal{B}_n$  for all $k=1,\dots,n.$
\end{proof} 

\vspace{.5cm}

By bringing together all of our main results, we
reveal a relationship between $C^*$-algebras generated by Toeplitz operators with vertical symbols and 
$C^*$-algebras generated by orthogonal projections acting on the poly-Bergman space $\mathcal{A}_n^2(\Pi)$.
Recall that $T_{n,a_0}: \mathcal{A}_n^2(\Pi) \rightarrow \mathcal{A}_n^2(\Pi)$  is the Toeplitz operator with vertical symbol $a_0(z)=\chi_{[0,\alpha/2]}(y)$,
where $\alpha>0$ is fixed. On the other hand, the orthogonal projection $P:  \mathcal{A}_n^2(\Pi) \rightarrow \mathcal{A}_n^2(\Pi)$ was constructed via an orthoprojection given in terms 
of Toeplitz operators acting on the Bergman space $\mathcal{A}^2(\Pi)$; see (\ref{vertical-aj}), (\ref{PT}), and (\ref{def-P}).

According to Corollary \ref{CAToeplitz} and Theorem \ref{C-P-Bj}, the $C^*$-algebra generated by all Toeplitz operators $T_{n,b}$, with vertical symbols satisfying $b^0=b^{+\infty}$, is 
contained in the $C^*$-algebra generated by the orthogonal projections
$P$, $B_{\Pi,(1)},..., B_{\Pi,(n)}$. In turn, this $C^*$-algebra is a $C^*$-subalgebra of 
$C^*(T_{n,a_0}, B_{\Pi,(1)},..., B_{\Pi,(n)})$, as follows from Theorem \ref{C-Ta,B(j)}.
The following diagram summarizes these relationships:
$$
\xymatrix{ 
	{\mathcal{C}}^*({\mathcal{T}}^{0,0}_n) \  \ar[d]^{\cong}  \subset & {\mathcal{C}}^*(P;\{B_{\Pi,(j)}\}_{j=1}^n) \ \subset \ar[d]^{\cong} & {\mathcal{C}}^*(T_{n,a_0}; \{B_{\Pi,(j)}\}_{j=1}^n) \ar[d]^{\cong} \\ 
	{\mathcal{D}}^{\mathbb{C}I}_n  & {\mathcal{D}}^{1,n}_n  & {\mathcal{D}}_n \\
}
$$

In addition,  the $C^*$-algebra $\mathcal{C}^* ({\mathcal{T}}^{0,\infty}_n)$ generated by all Toeplitz operators $T_{n,a}$, with $a \in  L_\infty^{\{0,\infty\}}(\mathbb{R}_+)$, is contained in ${\mathcal{C}}^*(T_{n,a_0};\{B_{\Pi,(j)}\}_{j=1}^n)$,
see Theorem \ref{CAToeplitz1}. 
Furthermore, for $a \in  L_\infty^{\{0,\infty\}}(\mathbb{R}_+)$, the vertical symbol $b_a(z)=a(z)+(a^{0}-a^{+\infty}) a_0(z)$
satisfies $b_a^{0}=b_a^{+\infty}$. Thus, we have 
$$\mathcal{C}^* ({\mathcal{T}}^{0,\infty}_n) = \mathcal{C}^* ({\mathcal{T}}^{0,0}_n)+\mathbb{C}\, T_{n,a_0}.$$
All the $C^*$-algebras here differ from each other just by
their irreducible one-dimensional representations corresponding to  $x=0,+\infty$. 
In this sense, the $C^*$-algebra generated by all Toeplitz operators with symbols in $L_\infty^{\{0,\infty\}}(\mathbb{R}_+)$ 
is closely related  to the $C^*$-algebra generated by a system of all-but-one orthogonal projections in generic position.


\end{document}